\documentclass{amsart}

\usepackage{fancyhdr}
\usepackage{lastpage}
\usepackage{yhmath}

\newcommand{\bdis}{\begin{displaymath}}
\newcommand{\edis}{\end{displaymath}}
\newcommand{\be}{\begin{equation}}
\newcommand{\ee}{\end{equation}}

\newcommand{\mcal}{\mathcal}

\newcommand{\vp}{\varphi}
\newcommand{\vt}{\vartheta}

\newcommand{\mT}{\mathring{T}}

\newcommand{\zf}{\zeta\left(\frac{1}{2}+it\right)}

\newtheorem{lemma}[]{Lemma}

\theoremstyle{definition}

\newtheorem{cor}[]{Corollary}

\theoremstyle{remark}
\newtheorem{remark}[]{Remark}

\newtheorem*{mydef11}{{\bf Theorem 1}}

\newtheorem*{mydef12}{{\bf Theorem 2}}

\numberwithin{equation}{section}

\begin{document}

\title{New consequences of the Riemann-Siegel formula and a law of asymptotic equality of signum-areas of $Z(t)$ function}

\author{Jan Moser}

\address{Department of Mathematical Analysis and Numerical Mathematics, Comenius University, Mlynska Dolina M105, 842 48 Bratislava, SLOVAKIA}

\email{jan.mozer@fmph.uniba.sk}

\keywords{Riemann zeta-function}

\begin{abstract}
In this paper we obtain the first mean-value theorems for the function $Z(t)$ on some disconnected sets. Next, we obtain a geometric law that
controls chaotic behavior of the graph of the function $Z(t)$. This paper is the English version of the papers \cite{8} and \cite{9}, except
of the Appendix that connects our results with the theory of Jacob's ladders, namely new third-order formulae have been obtained.
\end{abstract}

\maketitle

\section{Introduction}

\subsection{}

First of all we define the following collection of sequences

\bdis
\{ t_\nu(\tau)\},\ \nu=1,2,\dots,\ \tau\in [-\pi,\pi]
\edis
by the equation
\be \label{1.1}
\vt[t_\nu(\tau)]=\pi\nu+\tau;\ t_\nu(0)=t_\nu,
\ee
where (see \cite{12}, pp. 79, 329)
\be \label{1.2}
\begin{split}
 & Z(t)=e^{i\vt(t)}\zf , \\
 & \vt(t)=-\frac t2\ln\pi+\text{Im}\ln\Gamma\left(\frac 14+i\frac t2\right).
\end{split}
\ee
If we use this collection together with the Riemann-Siegel formula (see \cite{10}, p. 60, comp. \cite{12}, p. 79)
\be \label{1.3}
Z(t)=2\sum_{n\leq\bar{t}}\frac{1}{\sqrt{n}}\cos\{\vt(t)-t\ln n\}+\mcal{O}(t^{-1/4}),\ \bar{t}=\sqrt{\frac{t}{2\pi}},
\ee
then we obtain a new kind of mean-value theorems for the function $Z(t)$ on some disconnected sets. These formulae contain new information
about the distribution of positive and negative values of the function $Z(t)$ on corresponding sets.

\subsection{}

The main reason to introduce mentioned disconnected sets lies in the study of the internal structure of the Hardy-Littlewood estimate (1918),
(see \cite{1}, p. 178)
\be \label{1.4}
\int_T^{T+\Omega} Z(t){\rm d}t=o(\Omega)
\ee
for corresponding $\Omega$. \\

Let us remind the following about the estimate (\ref{1.4}). Hardy and Littlewood have obtained the following estimate (see \cite{1}, p. 178)
\be \label{1.5}
\int_T^{T+M} x(t){\rm d}t=\mcal{O}(T^\delta),\ M=T^{1/4+\epsilon},\quad \epsilon>0,\ \delta>0,
\ee
where (see \cite{1}, p. 177)
\bdis
\begin{split}
 & f(s)=\pi^{-s}e^{-\frac 12(s-\frac 14)\pi i}\Gamma(s)\zeta(2s),\quad s=\sigma+it, \\
 & x(t)=f\left(\frac 14+it\right).
\end{split}
\edis
Since (see \cite{1}, p. 178)
\bdis
x(t)=-\frac{e^{\frac 12\pi t}\Xi(2t)}{\frac 14+4t^2},
\edis
and (see \cite{12}, p. 68, (4.12.2); $\sigma=\frac 14$)
\bdis
\begin{split}
 & Z(t)=-2\pi^{1/4}\frac{\Xi(t)}{\left(\frac 14+t^2\right)\left|\Gamma\left(\frac 14+i\frac t2\right)\right|}, \\
 & \left|\Gamma\left(\frac 14+i\frac t2\right)\right|=
 t^{-1/4}e^{-\frac 12\pi t}\sqrt{2\pi}\left\{ 1+\mcal{O}\left(\frac 1t\right)\right\},
\end{split}
\edis
then we have the expression
\be \label{1.6}
x(t)=2^{1/2}\pi^{1/4}t^{-1/4}Z(2t)\left\{ 1+\mcal{O}\left(\frac 1t\right)\right\}.
\ee
Hence, we obtain from (\ref{1.5}) by (\ref{1.6}) the estimate
\be \label{1.7}
\int_T^{T+M} Z(2t){\rm d}t=\mcal{O}(T^{1/4+\delta}),
\ee
and, consequently
\bdis
\begin{split}
 & \int_{2T}^{2T+2M} Z(t){\rm d}t=\mcal{O}(T^{1/4+\delta});\quad 2T\to\bar{T},\ 2M\to \bar{M} \\
 & \int_{\bar{T}}^{\bar{T}+\bar{M}}Z(t){\rm d}t=\mcal{O}(\bar{T}^{1/4+\delta}),
\end{split}
\edis
i. e. we have (\ref{1.4}).

\section{New asymptotic formulae for the function $Z(t)$}

\subsection{}

Let
\be \label{2.1}
\begin{split}
 & G_1(x)=G_1(x;T,H)= \\
 & = \bigcup_{T\leq t_{2\nu}\leq T+H}\{ t:\ t_{2\nu}(-x)<t<t_{2\nu}(x)\},\ 0<x\leq\frac{\pi}{2}, \\
 & G_2(y)=G_2(y;T,H)= \\
 & = \bigcup_{T\leq t_{2\nu+1}\leq T+H}\{ t:\ t_{2\nu+1}(-y)<t<t_{2\nu+1}(y)\},\ 0<y\leq \frac{\pi}{2},
\end{split}
\ee
where
\be \label{2.2}
H=T^{1/6+2\epsilon};\ T^{1/6}\psi^2\ln^5T\to T^{1/6+\epsilon}.
\ee

The following theorem holds true.

\begin{mydef11}
\be \label{2.3}
\begin{split}
 & \int_{G_1(x)}Z(t){\rm d}t=\frac{2}{\pi}H\sin x+\mcal{O}(xT^{1/6+\epsilon}), \\
 & \int_{G_2(y)}Z(t){\rm d}t=-\frac{2}{\pi}H\sin y+\mcal{O}(yT^{1/6+\epsilon}), \\
 & x,y\in \left(\left. 0,\frac \pi 2\right]\right.
\end{split}
\ee
\end{mydef11}

\begin{remark}
By the formula (\ref{2.3}) we have expressed the internal structure of the Hardy-Littlewood integral
\bdis
\int_T^{T+H} Z(t){\rm d}t ,
\edis
i. e. the decomposition of this integral into its parts.
\end{remark}

\begin{remark}
We will assume that
\be \label{2.4}
G_1(x),G_2(y)\subset [T,T+H]
\ee
since we may put
\bdis
G_1(x)\cap [T,T+H]=\bar{G}_1(x)\to G_1(x), \dots
\edis
\end{remark}

\begin{remark}
The existence of the odd zero of the function
\bdis
Z(t),\ t\in \left[ T,T+T^{1/6+2\epsilon}\right]
\edis
follows directly from our formulae (\ref{2.3}).
\end{remark}

\subsection{}

Since (see (\ref{1.1}))
\be \label{2.5}
\begin{split}
 & \vt[t_{2\nu}(x)]-\vt[t_{2\nu}(-x)]=2x, \\
 & \vt[t_{2\nu+1}(y)]-\vt[t_{2\nu+1}(-y)]=2y,
\end{split}
\ee
we obtain (see \cite{1}, p. 102; \cite{3}, (42))
\be \label{2.6}
\begin{split}
 & t_{2\nu}(x)-t_{2\nu}(-x)=\frac{4x}{\ln\frac{T}{2\pi}}+\mcal{O}\left(\frac{xH}{T\ln^2T}\right), \\
 & t_{2\nu+1}(y)-t_{2\nu+1}(-y)=\frac{4y}{\ln\frac{T}{2\pi}}+\mcal{O}\left(\frac{yH}{T\ln^2T}\right).
\end{split}
\ee
Next we have (see (\ref{2.1}), (\ref{2.6}) and \cite{4}, (23))
\be \label{2.7}
m\{ G_1(x)\}=\frac x\pi H+\mcal{O}(x),\quad m\{ G_2(y)\}=\frac y\pi H+\mcal{O}(y),
\ee
where $m\{ G_1(x)\},m\{ G_2(y)\}$ stand for measures of the corresponding sets. Following (\ref{2.3}) we obtain

\begin{cor}
\be \label{2.8}
\begin{split}
 & \frac{1}{m\{ G_1(x)\}}\int_{G_1(x)}Z(t){\rm d}t\sim 2\frac{\sin x}{x}, \\
 & \frac{1}{m\{ G_2(y)\}}\int_{G_2(y)}Z(t){\rm d}t\sim -2\frac{\sin y}{y}.
\end{split}
\ee
\end{cor}

Since (see (\ref{1.1}), (\ref{2.1}))
\bdis
G_1(x)\cap G_2(y)=\emptyset ; \quad t_{2\nu}\left(\frac \pi 2\right)=t_{2\nu+1}\left(-\frac \pi 2\right) ,
\edis
we obtain from (\ref{2.3}) the following

\begin{cor}
\be \label{2.9}
\begin{split}
 & \int_{G_1(x)\cup G_2(y)} Z(t){\rm d}t=\\
 & \left\{\begin{array}{lcr} \frac{2}{\pi}(\sin x-\sin y)H+\mcal{O}\{ (x+y)T^{1/6+\epsilon}\} & , & x\not=y \\
 \mcal{O}(x T^{1/6+\epsilon}) & , & x=y. \end{array} \right.
\end{split}
\ee
\end{cor}

\begin{remark}
Since (see (\ref{2.4}))
\bdis
m\{ G_1(\pi/2)\}+m\{ G_2(\pi/2)\}=H,
\edis
we have (see (\ref{2.2}), (\ref{2.9}))
\bdis
\int_{G_1(\pi/2)\cup G_2(\pi/2)}Z(t){\rm d}t=\int_T^{T+H} Z(t){\rm d}t=o(H),
\edis
comp. (\ref{1.4}). Consequently, the estimate of the Hardy-Littlewood type is direct connects of the asymptotic formulae (\ref{2.3}).
\end{remark}

\section{Law of asymptotic equality of signum-areas of $Z(t)$ function}

\subsection{}

Let us point out the chaotic behavior of the graph of function $Z(t)$ at $t\to\infty$, (comp., for example, the graph of $Z(t)$ in the
neighborhood of the first Lehmer pair of the zeroes, \cite{2}, pp. 296, 297). In this direction we obtain a new law that controls this
chaotic behavior. \\

Let
\be \label{3.1}
\begin{split}
 & G_1^+(x)=\{ t:\ t\in G_1(x), Z(t)>0\}, \\
 & G_1^-(x)=\{ t:\ t\in G_1(x), Z(t)<0\}, \\
 & G_2^+(x)=\{ t:\ t\in G_2(x), Z(t)>0\}, \\
 & G_2^-(x)=\{ t:\ t\in G_2(x), Z(t)<0\}, \\
 & G_3(x)=\{ t:\ t\in G_1(x), Z(t)=0\}, \\
 & G_4(x)=\{ t:\ t\in G_2(x), Z(t)=0\} .
\end{split}
\ee
Of course,
\bdis
m\{ G_3(x)\}=m\{ G_4(x)\}=0.
\edis
The following Theorem holds true.

\begin{mydef12}
\be \label{3.2}
\begin{split}
 & \int_{G_1^+(x)\cup G_2^+(x)}Z(t){\rm d}t\sim -\int_{G_1^-(x)\cup G_2^-(x)}Z(t){\rm d}t, \\
 & T\to\infty,\ x\in \left(\left. 0,\frac{\pi}{2}\right]\right. .
\end{split}
\ee
\end{mydef12}

Let
\be \label{3.3}
\begin{split}
 & D^+(x)=\{ (t,u):\ t\in G_1^+(x)\cup G_2^+(x),\ 0<u\leq Z(t)\}, \\
 & D^-(x)=\{ (t,u):\ t\in G_1^-(x)\cup G_2^-(x),\ Z(t)\leq u<0\} .
\end{split}
\ee

\begin{remark}
The asymptotic equality (\ref{3.2}) expresses the following geometric law
\be \label{3.3}
m\{ D^+(x)\}\sim m\{ D^-(x)\},\quad T\to\infty,\ x\in \left(\left. 0,\frac{\pi}{2}\right.\right],
\ee
where $D^+(x),D^-(x)$ (see (\ref{3.3})) are the first mentioned signum-sets.
\end{remark}

\begin{remark}
It is just the geometric law (\ref{3.3}) that controls the chaotic behavior of the graph of function $Z(t)$.
\end{remark}

\section{Lemma 1}

The following lemma holds true.

\begin{lemma}
\be \label{4.1}
\sum_{T\leq t_\nu(\tau)\leq T+H} Z[t_\nu(\tau)]=\mcal{O}(T^{1/6+\epsilon}),\ \tau\in [-\pi,\pi].
\ee
\end{lemma}

\begin{proof}
Let us remind the formulae (see \cite{3}, (42); \cite{4}, (23))
\be \label{4.2}
\begin{split}
 & t_{\nu+1}-t_\nu=\frac{2\pi}{\ln\frac{T}{2\pi}}+\mcal{O}\left(\frac{H}{T\ln^2T}\right), \\
 & \sum_{T\leq t_\nu\leq T+H} 1=\frac{1}{2\pi}H\ln\frac{T}{2\pi}+\mcal{O}(1).
\end{split}
\ee
By the same way (comp. \cite{11}, p. 102; \cite{3}, (40) -- (42)) we obtain (see (\ref{1.1})) the formulae
\be \label{4.3}
\begin{split}
 & t_{\nu+1}(\tau)-t_\nu(\tau)=\frac{2\pi}{\ln\frac{T}{2\pi}}+\mcal{O}\left(\frac{H}{T\ln^2T}\right), \\
 & \sum_{T\leq t_\nu(\tau)\leq T+H} 1=\frac{1}{2\pi}H\ln\frac{T}{2\pi}+\mcal{O}(1).
\end{split}
\ee
Next, from (\ref{1.3}) we have
\bdis
\begin{split}
 & Z(t)=2\sum_{n<P_0}\frac{1}{\sqrt{n}}\cos\{ \vt-t\ln n\}+\mcal{O}(T^{-1/4}),\ P_0=\sqrt{\frac{T}{2\pi}}, \\
 & t\in [T,T+H]
\end{split}
\edis
and, consequently (see (\ref{1.1}))
\be \label{4.4}
\begin{split}
 & Z[t_\nu(\tau)]=2(-1)^\nu\cos\tau+ \\
 & +2(-1)^\nu\sum_{2\leq n<P_0}\frac{1}{\sqrt{n}}\cos\{ \tau-t_\nu(\tau)\ln n\}+\mcal{O}(T^{-1/4}).
\end{split}
\ee
Hence, (see (\ref{4.3}), (\ref{4.4}))
\be \label{4.5}
\begin{split}
 & \sum_{T\leq t_\nu(\tau)\leq T+H} Z[t_\nu(\tau)]= \\
 & = 2\cos\tau\sum_{2\leq n<P_0}\frac{1}{\sqrt{n}}\sum_{T\leq t_\nu(\tau)\leq T+H}(-1)^\nu\cos\{ t_\nu(\tau)\ln n\}+ \\
 & + 2\sin\tau \sum_{2\leq n<P_0}\frac{1}{\sqrt{n}}\sum_{T\leq t_\nu(\tau)\leq T+H}(-1)^\nu\cos\{ t_\nu(\tau)\ln n\}+\mcal{O}(\ln T)= \\
 & = 2w_1\cos\tau+2w_2\sin\tau+\mcal{O}(\ln T).
\end{split}
\ee
By the method of the papers \cite{3}, \cite{7} we obtain the estimate
\be \label{4.6}
w_1=\mcal{O}(T^{1/6+\epsilon}).
\ee
Next, instead of \cite{3}, (54) we have by \cite{5}, (66)
\bdis \label{54p}  \tag{54'}
\begin{split}
 & \sum_{T\leq t_\nu\leq T+H} (-1)^\nu\sin\{ t_\nu\ln n\}= \\
 & = \frac 12 (-1)^{\bar{\nu}}\sin\vp+\frac 12 (-1)^{N+\bar{\nu}}\sin(\omega N+\vp)-\frac 12(-1)^{\bar{\nu}}\tan\frac{\omega}{2}\cos\vp- \\
 & - \frac 12 (-1)^{N+\bar{\nu}}\tan\frac{\omega}{2}\cos(\omega N+\vp)+\mcal{O}\left( \frac{H^3\ln n}{T}\right),
\end{split}
\edis
and instead of \cite{7}, (42) we have
\bdis \label{42p}  \tag{42'}
\begin{split}
 & \cos\vp-\cos(\omega N+\vp)=2\sin\frac{\omega N}{2}\sin\left(\frac{\omega N}{2}+\vp\right)= \\
 & = -2\sin(Nx(n))\sin(\tilde{T}\ln n).
\end{split}
\edis
Then, by the method of the papers \cite{3},\cite{7} we obtain the estimate
\be \label{4.7}
w_2=\mcal{O}( T^{1/6+\epsilon}).
\ee
Hence, the estimate (\ref{4.1}) follows from (\ref{4.5}) by (\ref{4.6}), (\ref{4.7}).
\end{proof}

\section{Lemma 2}

Next, the following lemma holds true.

\begin{lemma}
\be \label{5.1}
\begin{split}
 & \sum_{T\leq t_\nu(\tau)\leq T+H}(-1)^\nu Z[t_\nu(\tau)]= \\
 & = \frac 1\pi H\ln\frac{T}{2\pi}\cos\tau+\mcal{O}(T^{1/6}\ln T),\quad \tau\in [-\pi,\pi].
\end{split}
\ee
\end{lemma}

\begin{remark}
The formula (\ref{5.1}) is the asymptotic formula for
\bdis
H=T^{1/6+2\epsilon},\quad \tau\in [-\pi,\pi],\ \tau\not=-\frac{\pi}{2},\frac{\pi}{2}.
\edis
\end{remark}

\begin{proof}
We have the following formula from (\ref{4.4})
\be \label{5.2}
\begin{split}
 & \sum_{T\leq t_\nu(\tau)\leq T+H} (-1)^\nu Z[t_\nu(\tau)]= \frac 1\pi H\ln\frac{H}{2\pi}\cos\tau +\\
 & + 2\cos\tau\sum_{2\leq n<P_0}\frac{1}{\sqrt{n}}\sum_{T\leq t_\nu(\tau)\leq T+H} \cos\{ t_\nu(\tau)\ln n\}+ \\
 & + 2\sin\tau\sum_{2\leq n<P_0}\frac{1}{\sqrt{n}}\sum_{T\leq t_\nu(\tau)\leq T+H} \sin\{ t_\nu(\tau)\ln n\}+\mcal{O}(\ln T)= \\
 & = \frac 1\pi H\ln\frac{T}{2\pi}\cos\tau+2w_3\cos\tau+2w_4\sin\tau+\mcal{O}(\ln T)= \\
 & = \frac 1\pi H\ln\frac{T}{2\pi}\cos\tau+\mcal{O}(T^{1/6}\ln T),
\end{split}
\ee
where the estimates for the sums $w_3,w_4$ were obtained by the methods of papers \cite{4},\cite{6}.
\end{proof}

\section{Proof of Theorem 1}

First of all, we have the following formulae (see (\ref{4.1}), (\ref{5.1}))
\be \label{6.1}
\begin{split}
 & \sum_{T\leq t_{2\nu}(\tau)\leq T+H} Z[t_{2\nu}(\tau)]=\frac{1}{2\pi}H\ln\frac{T}{2\pi}\cos\tau+\mcal{O}(T^{1/6+\epsilon}), \\
 & \sum_{T\leq t_{2\nu+1}(\tau)\leq T+H} Z[t_{2\nu+1}(\tau)]=-\frac{1}{2\pi}H\ln\frac{T}{2\pi}\cos\tau+\mcal{O}(T^{1/6+\epsilon}).
\end{split}
\ee
Since (see \cite{11}, p. 100)
\bdis
\vt'(t)=\frac 12\ln\frac{t}{2\pi}+\mcal{O}\left(\frac 1t\right),
\edis
we obtain from (\ref{1.1}) that
\bdis
\left(\frac{{\rm d}t_\nu(\tau)}{{\rm d}\tau}\right)^{-1}=\ln P_0+\mcal{O}\left(\frac HT\right),\quad
t_\nu(\tau)\in [T,T+H].
\edis
Next (see (\ref{2.6}))
\be \label{6.2}
\begin{split}
 & \int_{-x}^x Z[t_{2\nu}(\tau)]{\rm d}\tau=\int_{-x}^x Z[t_{2\nu}(\tau)]\frac{{\rm d}t_{2\nu}(\tau)}{{\rm d}\tau}
 \left(\frac{{\rm d}t_{2\nu}(\tau)}{{\rm d}\tau}\right)^{-1}= \\
 & = \ln P_0\int_{t_{2\nu}(-x)}^{t_{2\nu}(x)} Z(t){\rm d}t+\mcal{O}(HT^{-5/6}),
\end{split}
\ee
where
\be \label{6.3}
Z(t)=\mcal{O}(t^{1/6}\ln t),\quad t\to\infty.
\ee
Of course, (see (\ref{4.2}), (\ref{4.3}) and (\ref{6.3}))
\be \label{6.4}
\begin{split}
 & \sum_{T\leq t_{2\nu}(\tau)\leq T+H} Z[t_{2\nu}(\tau)]=\sum_{T\leq t_{2\nu}\leq T+H} Z[t_{2\nu}(\tau)]+\mcal{O}(T^{1/6}\ln T).
\end{split}
\ee
Finally, by integration (see (\ref{6.2})) of the first formula in (\ref{6.1}) (after the transformation (\ref{6.4})) we obtain the first formula
in (\ref{2.3}). The second formula in (\ref{2.3}) can be obtained by the similar way.

\section{Proof of Theorem 2}

The following holds true (see (\ref{2.2}), (\ref{2.9}), (\ref{3.1}))
\be \label{7.1}
\begin{split}
 & \int_{G_1(x)\cup G_2(x)}Z(t){\rm d}t=\int_{G_1^+(x)\cup G_1^-(x)\cup G_2^+(x)\cup G_2^-(x)} Z(t){\rm d}t= \\
 & = \int_{G_1^+(x)\cup G_2^+(x)} Z(t){\rm d}t+\int_{G_1^-(x)\cup G_2^-(x)} Z(t){\rm d}t= \\
 & = \mcal{O}(xT^{1/6+\epsilon})=o(xH).
\end{split}
\ee
Since (see (\ref{2.3}))
\bdis
\int_{G_1(x)} Z(t){\rm d}t>\left(\frac 2\pi\sin x-\epsilon\right)H=A(x,\epsilon)H,\quad 0<\epsilon<\frac 1\pi\sin x
\edis
and
\bdis \begin{split}
& \int_{G_1(x)} Z(t){\rm d}t=\int_{G_1^+(x)\cup G_1^-(x)} Z(t){\rm d}t\leq \int_{G_1^+(x)} Z(t){\rm d}t\leq \\
& \leq \int_{G_1^+(x)\cup G_2^+(x)}Z(t){\rm d}t,
\end{split}
\edis
we have the following inequality
\be \label{7.2}
\int_{G_1^+(x)\cup G_2^+(x)} Z(t){\rm d}t> A(x,\epsilon)H
\ee
and, by the similar way, we have
\be \label{7.3}
-\int_{G_1^-(x)\cup G_2^-(x)} Z(t){\rm d}t >B(x,\epsilon)H.
\ee
Hence, the equality (see (\ref{7.1}))
\bdis
\begin{split}
 & \int_{G_1^+(x)\cup G_2^+(x)} Z(t){\rm d}t= \\
 & =-\int_{G_1^-(x)\cup G_2^-(x)} Z(t){\rm d}t+o(xH),\quad x\in \left(\left. 0,\frac{\pi}{2}\right]\right.
\end{split}
\edis
is the asymptotic equality (\ref{3.2}) by (\ref{7.2}) and (\ref{7.3}).

\appendix

\section{Jacob's ladders and new third-order formulae corresponding to (\ref{2.8}) and (\ref{3.2})}

In our paper \cite{13}, (9.2), (9.5) we have proved the following lemma: if
\bdis
\vp_1\{ [\mT,\widering{T+U}]\}=[T,T+U],
\edis
then for every Lebesgue-integrable function
\bdis
f(x),\ x\in [T,T+U]
\edis
we have
\be \label{A1}
\begin{split}
 & \int_{\mT}^{\widering{T+U}}f[\vp_1(t)]\tilde{Z}^2(t){\rm d}t=\int_T^{T+U} f(x){\rm d}x, \\
 & T\geq T_0[\vp_1],\ U\in\left(\left. 0,\frac{T}{\ln T}\right.\right],
\end{split}
\ee
where
\be \label{A2}
\begin{split}
 & \tilde{Z}^2(t)=\frac{Z^2(t)}{\left\{ 1+\mcal{O}\left( \frac{\ln\ln t}{\ln t}\right)\right\}\ln t}=\omega(t)Z^2(t); \\
 & \omega(t)=\frac{1}{\left\{ 1+\mcal{O}\left( \frac{\ln\ln t}{\ln t}\right)\right\}\ln t}=\frac{1}{\ln t}
 \left\{1+\mcal{O}\left( \frac{\ln\ln t}{\ln t}\right)\right\},
\end{split}
\ee
and $\vp_1(t)$ is a fixed Jacob's ladder. Consequently we have (see (\ref{A1}), (\ref{A2}))
\be \label{A3}
\int_{\mT}^{\widering{T+U}} \omega(t)f[\vp_1(t)]Z^2(t){\rm d}t=\int_T^{T+U} f(x){\rm d}x,\quad U\in\left(\left. 0,\frac{T}{\ln T}\right.\right].
\ee
Now, we obtain from (\ref{2.8}), (\ref{3.2}) by (\ref{A3}) the following third-order formulae
\bdis \tag{2.8'}
\begin{split}
 & \frac{1}{m\{ G_1(x)\}}\int_{\mathring{G}_1(x)}\omega(t)Z[\vp_1(t)]Z^2(t){\rm d}t\sim 2\frac{\sin x}{x}, \\
 & \frac{1}{m\{ G_2(y)\}}\int_{\mathring{G}_2(y)}\omega(t)Z[\vp_1(t)]Z^2(t){\rm d}t\sim -2\frac{\sin y}{y},
\end{split}
\edis
and
\bdis \tag{3.2'}
\begin{split}
 & \int_{\mathring{G}_1^+(x)\cup \mathring{G}_2^+(x)}\omega(t)Z[\vp_1(t)]Z^2(t){\rm d}t\sim \\
 & \sim -\int_{\mathring{G}_1^-(x)\cup\mathring{G}_2^-(x)}\omega(t)Z[\vp_1(t)]Z^2(t){\rm d}t,\quad T\to\infty.
\end{split}
\edis

\thanks{I would like to thank Michal Demetrian for helping me with the electronic version of this work.}

\end{document}